\documentclass[leqno,12pt]{amsart} 
\usepackage{color}
\usepackage{amsmath}    
\usepackage{amssymb}  
\usepackage{amscd}
\usepackage{amsthm}  
\usepackage{graphicx}
\usepackage{comment}	
\usepackage{txfonts} 	
\usepackage[dvipdfm,
  colorlinks=false,
  bookmarks=true,
  bookmarksnumbered=false,
  pdfborder={0 0 0},
  bookmarkstype=toc]{hyperref} 

\numberwithin{equation}{section}
\theoremstyle{plain} 
\newtheorem{thm}{\indent\sc Theorem}[section] 
\newtheorem{lem}[thm]{\indent\sc Lemma}
\newtheorem{cor}[thm]{\indent\sc Corollary} 

\newtheorem{claim}[thm]{\indent\sc Claim}

\theoremstyle{definition}
\newtheorem{dfn}[thm]{\indent\sc Definition}
\newtheorem{exa[thm]}{\indent\sc Example}

\newtheorem{rem}[thm]{\indent\sc Remark}
\newtheorem{prob}[thm]{\indent\sc Problem}

\newcommand{\kah}{K\"{a}hler }

\newcommand{\deldel}{\sqrt{-1}\partial \overline{\partial}}

\begin{document}
\baselineskip=16pt

\title[Nadel-Nakano vanishing theorems of vector bundles with singular Hermitian metrics]
{Nadel-Nakano vanishing theorems of vector bundles with singular Hermitian metrics}
\author{Masataka Iwai}

\address{Graduate School of Mathematical Sciences, The University of Tokyo, 3-8-1 Komaba,
Tokyo, 153-8914, Japan.}

 \email{{\tt
masataka@ms.u-tokyo.ac.jp}}



\maketitle
\begin{abstract}
We study a singular Hermitian metric of a vector bundle.
First, we prove the sheaf of locally square integrable holomorphic sections of a vector bundle
with a singular Hermitian metric, which is a higher rank analogy of a multiplier ideal sheaf, 
is coherent under some assumptions.
Second, we prove a Nadel-Nakano type vanishing theorem 
of a vector bundle
with a singular Hermitian metric.
We do not use an approximation technique of a singular Hermitian metric. 
We apply these theorems to a singular Hermitian metric induced by holomorphic sections and
a big vector bundle,
and we obtain a generalization of Griffiths' vanishing theorem. 
Finally, we show a generalization of Ohsawa's vanishing theorem.
\end{abstract}

\section{Introduction}
The aim of this paper is to study the vanishing theorem of a vector bundle with a singular Hermitian metric.
Here is a brief history of a singular Hermitian metric of a vector bundle.
A singular Hermitian metric of a vector bundle is a higher rank analogy of
 a singular Hermitian metric of a line bundle. 
The origin is due to de Cataldo \cite{Cat98}.
After that, Berndtsson and P\u{a}un defined a singular Hermitian metric of a vector bundle in the different way in \cite{BP}.
We adopt the definition of a singular Hermitian metric of a vector bundle in \cite{BP}.
They also defined the notions of a positivity of a singular Hermitian metric,
called {\it positively curved}.
In \cite{PT14}, P\u{a}un and Takayama proved that 
a direct image sheaf of
an $m$-th relative canonical line bundle $f_{*}(m K_{X/Y})$
can be endowed with a positively curved singular Hermitian metric
for any fibration $f \colon X \rightarrow Y$.
Recently by using this result, Cao and P\u{a}un \cite{CP} proved Iitaka conjecture 
when the base space is an Abelian variety. 
For more details, we refer the reader to \cite{Paun16}.

Although a singular Hermitian metric of a vector bundle 
was investigated in many papers
(for example \cite{BP}, \cite{PT14}, \cite{Hos17}, \cite{HPS17}, \cite{Raufi}, etc.),
there exist few results of vanishing theorems of vector bundles with singular Hermitian metrics.
We explain the details below.
Let $( X , \omega) $ be a compact \kah manifold and $( E, h )$ 
be a vector bundle with a singular Hermitian metric.
In \cite{Cat98},
{\it the sheaf of locally square integrable holomorphic sections of E with respect to h},
denoted by $E(h)$, is defined to be 
$$
E(h)_x = \{ f_x \in E_{(x)}
  \colon  |f_x |_{h}^2 \in L^1_{loc} \}
\,\,\,
x \in X,
$$
which is a higher rank analogy of a multiplier ideal sheaf.
In this paper, we will denote by $E_{(x)}$ the stalk of $E$ at $x$, defined by 
$\displaystyle \varinjlim_{x \in U} H^0(U , E)$.
We consider the following problems.
\begin{prob}
\label{problem}
\begin{enumerate}
\item Is $E(h)$ a coherent sheaf ?
\item Does there exist a Nadel-Nakano type vanishing theorem, that is, the vanishing of the cohomology group $H^q( X , K_X \otimes E(h) ) $ for any $q \ge 1$ if $h$ has some positivity ?
\end{enumerate}

\end{prob}
Unlike a multiplier ideal sheaf, we do not know if $E(h)$ is coherent.
In \cite{Cat98}, de Cataldo proved $E(h)$ is coherent and a Nadel-Nakano type vanishing theorem
if $h$ has an approximate sequence of smooth Hermitian metrics $\{  h_{\mu} \}$
satisfying
$h_{\mu}  \uparrow h$ pointwise and $\sqrt{-1} \Theta_{E , h_{\mu}} - \eta \omega \otimes Id_{E} \ge 0$ in the sense of Nakano 
for some positive and continuous function $\eta$.
However, $h$ does not always have such approximate sequence 
(See in \cite[Example 4.4]{Hos17} ).
Therefore these problems are open.


We give a partial answer to Problem \ref{problem}.
First we prove the coherentness of $E(h)$ under some assumptions.
\begin{thm}
\label{coh}
Let $(X , \omega)$ be a \kah manifold and $(E , h)$ be a holomorphic vector bundle on $X$ with a singular Hermitian metric.
We assume the following conditions.
\begin{enumerate}
\item There exists a proper analytic subset $Z$ such that $h$ is smooth on $X  \setminus Z$.
\item $h e^{-\zeta}$ is a positively curved singular Hermitian metric on $E$
for some continuous function $\zeta$ on $X$.
\item There exists a real number $C$ such that $\sqrt{-1} \Theta_{E , h} - C \omega \otimes \textit{Id}_{E}  \ge 0$ on $X  \setminus Z$ in the sense of Nakano.
\end{enumerate}
Then the sheaf $E(h)$ is coherent.
\end{thm}

Next we study the cohomology group $H^q( X , K_X \otimes E(h) )$ for any $q \ge 1$.
We prove a vanishing theorem and an injectivity theorem of vector bundles with singular Hermitian metrics under some assumptions.
\begin{thm}
\label{vanishing}
Let $(X , \omega)$ be a compact \kah manifold and $(E , h)$ be a holomorphic vector bundle on $X$ with a singular Hermitian metric.
We assume the following conditions.
\begin{enumerate}
\item There exists a proper analytic subset $Z$ such that $h$ is smooth on $X  \setminus Z$.
\item $h e^{-\zeta}$ is a positively curved singular Hermitian metric on $E$
for some continuous function $\zeta$ on $X$.
\item There exists a positive number $\epsilon > 0$ such that $\sqrt{-1} \Theta_{E , h} - \epsilon \omega \otimes \textit{Id}_{E}  \ge 0$ on $X  \setminus Z$ in the sense of Nakano.
\end{enumerate}
Then $H^q( X , K_X \otimes E(h) ) = 0$ holds for any $q \ge 1$.
\end{thm}
\begin{thm}

\label{injective}
Let $(X , \omega)$ be a compact \kah manifold, $(E , h)$ be a holomorphic vector bundle on $X$ with a singular Hermitian metric and
$(L , h_L )$ be a holomorphic line bundle with a smooth metric.
We assume the following conditions.
\begin{enumerate}
\item There exists a proper analytic subset $Z$ such that $h$ is smooth on $X  \setminus Z$.
\item $h e^{-\zeta}$ is a positively curved singular Hermitian metric on $E$
for some continuous function $\zeta$ on $X$.
\item $\sqrt{-1} \Theta_{E , h}  \ge 0$ on $X  \setminus Z$ in the sense of Nakano.
\item There exists a positive number 
$\epsilon > 0$ such that $\sqrt{-1} \Theta_{E , h} - \epsilon \sqrt{-1} \Theta_{L , h_L} \otimes \textit{Id}_{E}  \ge 0$ on $X  \setminus Z$ in the sense of Nakano.

\end{enumerate}

Let $s$ be a non zero section of $L$.
Then for any $q \ge 0$, the multiplication homomorphism 
$$
\times s  : H^q( X , K_X \otimes E(h) ) \rightarrow H^q( X , K_X \otimes L \otimes E(h) )
$$
is injective.
\end{thm}
Therefore we proved a Nadel-Nakano type vanishing theorem with some assumptions.
If $E$ is a holomorphic line bundle, these theorems were proved in \cite{Fuj07}.
We point out we do not use an approximation sequence of a singular Hermitian metric to show these theorems.

Some applications are indicated as follows.
First we treat a singular Hermitian metric induced by holomorpic sections, proposed by Hosono \cite[Chapter 4]{Hos17}.
By caluculating the curvature of this metric,
we prove that
we can apply Theorem \ref{vanishing} 
to Hosono's example.
Therefore we can apply a Nadel-Nakano type vanishing theorem
even if $h$ does not have an approximate sequence such as \cite{Cat98}.
Second, we generalize Griffiths' vanishing theorem,
that is, 
$H^q(X , K_X \otimes {\rm Sym}^{m}(E) \otimes \det E) = 0$ holds
for any $m \ge 0$ and $q \ge 1$ if $E$ is an ample vector bundle.
We treat the case when $E$ is a big vector bundle.
If $E$ is a big vector bundle with some assumptions, ${\rm Sym}^{m}(E) \otimes \det E$ can be endowed
with a singular Hermitian metric $h_m$
satisfying the assumptions such as Theorem \ref{vanishing}
(see Section 5.2).
Therefore
$H^q(X , K_X \otimes ({\rm Sym}^{m}(E) \otimes \det E) (h_m)) = 0$ holds
for any $m \ge 0$ and $q \ge 1$.

Finally, we generalize Ohsawa's vanishing theorem.
\begin{thm}
\label{ohsawa}
Let $(X , \omega)$ be a compact \kah manifold and $(E , h)$ be a holomorphic vector bundle on $X$ with a singular Hermitian metric.
Let $\pi \colon X \rightarrow W$ be a proper surjective holomorphic map to an analytic space with a \kah form $\sigma$.
We assume the following conditions.
\begin{enumerate}
\item There exists a proper analytic subset $Z$ such that $h$ is smooth on $X  \setminus Z$.
\item $h e^{-\zeta}$ is a positively curved singular Hermitian metric on $E$
for some continuous function $\zeta$ on $X$.
\item $\sqrt{-1} \Theta_{E , h} -  \pi^{*}\sigma \otimes \textit{Id}_{E}  \ge 0$ on $X  \setminus Z$ in the sense of Nakano.
\end{enumerate}
Then $H^q\bigl( W  , \pi_{*} (K_X \otimes E(h) ) \bigr) = 0$ holds for any $q \ge 1$.
\end{thm}
If $h$ is smooth, this theorem was proved by Ohsawa \cite{Oh88}.

The organization of this paper is as follows.
In Section 2, we review some of standard facts on
vector bundles, singular Hermitian metrics and $L^2$ estimates.
In Section 3, we prove Theorem \ref{coh}.
The proof is based on \cite[Lemma 5]{Dem2} and \cite[Proposition 4.1.3]{Cat98}.
Since $h$ has singularities along $Z$, we apply the $L^2$ estimate only outside $Z$.
In Section 4, we prove Theorem \ref{vanishing} and \ref{injective}.
Based on \cite[Claim 1]{Fuj07},
we prove the cohomology isomorphism between the cohomology group of $K_X \otimes E(h) $ on $X$
and the $L^2$ cohomology group on $Y$ by using \v{C}ech cohomology.
From this isomorphism, it is easy to prove Theorem \ref{vanishing} and \ref{injective}.
In Section 5, we treat some applications.
We treat a singular Hermitian metric induced by holomorpic sections
and prove a generalization of Griffiths' vanishing theorem.
In Section 6, we prove a generalization of Ohsawa's vanishing theorem by using the methods of  \cite[Theorem 3.1]{Oh88} and Section 4.

{\bf Acknowledgment. } 
The author would like to thank his supervisor Prof. Shigeharu Takayama for helpful comments and enormous support.
He would like to thank Genki Hosono and Takahiro Inayama for useful comments
about the applications in Section 5.
This work was supported by the Grant-in-Aid for Scientific
Research (KAKENHI No.20284333) and the Grant-in-Aid for JSPS fellows. 
This work was supported by the Program for Leading Graduate Schools, MEXT, Japan.

\section{Preliminary}
\subsection{Hermitian metrics on vector bundles}

We briefly explain definitions and notations of smooth Hermitian metrics of vector bundles.

We will denote by $(X , \omega)$ a compact \kah manifold and denote by $E$ 
a holomorphic vector bundle of rank $r$ on $X$. 
For any point $x \in X$, we take a system of local coordinate $( V ; z_1 , \dots , z_n)$
near $x$. 
Let $h$ be a smooth metric on $E$ and $e_1, \dots ,e_r$ be a local orthogonal holomorphic frame of $E$ near $x$.
We denote by
$$
\sqrt{-1} \Theta_{E , h} = 
\sqrt{-1} \sum_{1 \le j,k  \le n , \, 1 \le \lambda , \mu \le r}
c_{jk \lambda \mu} \,
dz_j \wedge d \bar{z}_k
\otimes
e_{\lambda}^{*} \otimes e_{\mu}
$$
the Chern curvature tensor.
For any
$u = \sum_{1 \le j  \le n , \, 1 \le \lambda \le r}u_{j \lambda}dz_j \otimes e_{\lambda} \in T_{x}X \otimes E_x$, we denote by
$$
\theta_{E , h}(u) =
 \sum_{1 \le j,k  \le n , \, 1 \le \lambda , \mu \le r}
c_{jk \lambda \mu} u_{j \lambda} \bar{u}_{k \mu}
$$
and
$$
\theta_{ \omega \otimes id_{E} }(u)
=
\sum_{1 \le j,k  \le n , \, 1 \le \lambda \le r}
\omega_{jk} u_{j \lambda} \bar{u}_{k \lambda},
$$
where $\omega = \sqrt{-1} \sum_{1 \le j, k \le n} \omega_{jk} dz_j \wedge d \bar{z}_k$.

\begin{dfn}\cite[Chapter 7 \S 6]{Dem}
For any real number $C$,
we write $\sqrt{-1} \Theta_{E , h} \ge C \omega \otimes id_{E} $ {\it in the sense of Nakano} if
$\theta_{E , h}(u) - C \theta_{ \omega \otimes id_{E} }(u) \ge 0$ for any $u \in TX \otimes E$.
\end{dfn}

We review definitions of singular Hermitian metrics.
For more details, we refer the reader to \cite[Section 2]{Paun16}.
Let $H_r$ be the set of $r \times r$ semipositive definite Hermitian matrixs
and $\bar{H}_r$ be the space of semipositive, possibly unbounded Hermitian forms on
$\mathbb{C}^r$.

\begin{dfn} \cite[Definiton 2.8 and Definition 2.9]{Paun16}
\begin{enumerate}
\item The {\it singular Hermitian metric} $h$ on $E$ is defined to be a locally measurable map with
values in $\bar{H}_r$ such that $0 < \det h < +\infty$ almost everywhere.
\item A singular Hermitian metric $h$ on $E$ is said to be {\it negatively curved} if the function
$\log |v|^{2}_{h}$ is plurisubharmonic for any local section $v$ of $E$.
\item A singular Hermitian metric $h$ on $E$ is said to be {\it positively curved} if the dual
singular Hermitian metric $h^{*} = {}^t\! h^{-1}$ on the dual vector bundle $E^{*}$ is negatively
curved.
\end{enumerate}
\end{dfn}

We prove the following lemma of a positively curved singular Hermitian metric.

\begin{lem}
\label{keylemma}
For any point $x \in X$, we take a system of local coordinate $( V ; z_1 , \dots , z_n)$
near $x$ and take a local holomorphic frame $e_1, \dots ,e_r$ of $E$ on $V$.
Let $U \subset \subset V$ be an open set near $x$.
We assume there exists a continuous function $\zeta$ on $X$ such that $h e^{-\zeta}$ is a
positively curved singular Hermitian metric on $E$.
Then there exists a positive number $M_U$ such that for any $u \in H^0( V , E)$
$$
| u |^{2}_{h} \ge M_U \sum_{1 \le i \le r} |u_i|^2
$$
holds on $U$, where $u = \sum_{1 \le i \le r} u_ie_i$.
\end{lem}

\begin{proof}
We may assume $u = u_1e_1$.
By \cite[Chapter 16]{HPS17}, we obtain 
$$ 
 | u |_{he^{-\zeta}} (z) = \sup_{  f \in E^{*}_{z} } \frac{ | f(u) |(z) }{   | f |_{ (he^{-\zeta} )^{*} } } 
\ge
\frac{  |e^{*}_1 (u)  |(z)}{   |e^{*}_1 |_{ (h e^{-\zeta} )^{*} }  } 
=
\frac{  |u_1| (z) }{   |e^{*}_1|_{ (h e^{-\zeta} )^{*} }  }
$$
for any $z \in V$.
Since $h e^{-\zeta}$ is positively curved, $ |e^{*}_1 |_{ (h e^{-\zeta} )^{*} } $ is a plurisubharminic function
on $V$.
Therefore $ |e^{*}_1 |_{ (h e^{-\zeta} )^{*} } $ is bounded above on $U$.
We take a positive number $M_1$ such that $ |e^{*}_1 |_{ (h e^{-\zeta} )^{*} }  \le M_1$,
then we have $| u |_{he^{-\zeta}} \ge \frac{ |u_1| }{ M_1}$.
Since $e^{\zeta}$ is a positive continuous function, we can take a positive number $M$ such that 
$ e^{\zeta} \ge M$ on $X$. 
We put $M_{U} \coloneqq \frac{M^2}{M^{2}_1}$ and we obtain
$$
 | u |^{2}_{h}
 = 
 | u |^{2}_{he^{-\zeta}} e^{2\zeta}
 \ge
  M_{U} |u_1|^2,
$$
which completes the proof.
\end{proof}

\subsection{$L^2$ estimates and harmonic integrals on complete \kah manifolds}

We need an $L^2$ estimate on a complete \kah manifold.
Let $Y$ be a complete \kah manifold, $\omega'$ be a (not necessarily complete) \kah form
and $(E , h)$ be a vector bundle with a smooth Hermitian metric.
The $L^2$ space 
$L^{2}_{ n , q }( Y , E )_{ \omega' , h} $
is defined by the set of $ E $-valued $(n,q)$ forms with measurable coefficients on  Y 
such that $\int_{Y} | f |^2_{\omega' , h } dV_{ \omega'} < + \infty $,
where $dV_{\omega'} \coloneqq \omega'^{n}/n!$ is a volume form on $Y$.
\begin{thm}
\cite[Chapter 7  \S 7 and Chapter 8 \S 6]{Dem}
\cite[Lemme 3.2 and Th\'eor\`eme 4.1]{Dem82}
\label{estimate}
Under the conditions stated above,
we also assume that there exists a positive number $\epsilon > 0$ such that 
$\sqrt{-1}\Theta_{E , h} \ge \epsilon \omega' \otimes \textit{Id}_{E} $ in the sense of Nakano.
Then for any $q \ge 1$ and any $g \in L^{2}_{ n , q }( Y , E )_{ \omega', h}$ such that $\bar{\partial}g = 0$,
there exists $f \in L^{2}_{ n , q-1 }( Y , E )_{ \omega' , h}$  such that $\bar{\partial} f =  g $ and 
$$
\int_{Y} | f |^2_{\omega' , h } dV_{ \omega'}  \le 
\frac{1}{q \epsilon}
\int_{Y} | g|^2_{\omega', h } dV_{ \omega' }.
$$

\end{thm}

We use a fact of harmonic integrals to prove Theorem \ref{injective}.
For more details, we refer the reader to \cite[Section 2]{Fuj07} or \cite[Chapter 8]{Dem}.
The maximal closed extension of the $\bar{\partial}$ operator determines a densely
defined closed operator
$\bar{\partial} \colon L^{2}_{ n , q }( Y , E )_{ \omega' , h} \rightarrow L^{2}_{ n , q + 1}( Y , E )_{ \omega' , h} $.
Then we obtain the following orthogonal decomposition.
\begin{thm}\cite[Section 3]{Fuj07}, \cite[Chapter 8]{Dem}.
\label{harmonic}
$$
L^{2}_{ n , q }( Y , E )_{ \omega' , h}  =
\overline{{\rm Im}\bar{\partial}} \oplus
\mathcal{H}^{n,q}( Y , E ) \oplus
\overline{{\rm Im} \bar{\partial}^{*}_{\omega' , h}}
$$
holds, where $\bar{\partial}^{*}_{\omega' , h}$ is the Hilbert space adjoint of $\bar{\partial}$
and $\mathcal{H}^{n,q}( Y , E ) $ is the set of harmonic forms
defined by
$$
\mathcal{H}^{n,q}( Y , E ) \coloneqq 
\{ f \in L^{2}_{ n , q }( Y , E )_{ \omega' , h} \colon \bar{\partial}f =  \bar{\partial}^{*}_{\omega' , h}f = 0 \}.
$$
\end{thm}

\section{Coherentness of $E(h)$}

We prove Theorem \ref{coh}.

\begin{proof}
We may assume that $X$ is a unit ball in $\mathbb{C}^n$, $E = X \times \mathbb{C}^r$, and $\omega$ is a standard Euclidean metric.
Let $e_1, \dots ,e_r$ be a local holomorphic frame of $E$ on $X$.
We take a open ball $U \subset \subset X$.
It is enough to show that there exists a coherent sheaf $\mathcal{F}$ on $U$ such that $E(h)_x = \mathcal{F}_x$ 
for any $x \in U$.

We will denote by $\mathcal{G}$ the space of holomorphic sections $g \in H^0(U , E)$
such that $\int_{U} | g |^{2}_{h} dV_{\omega} < \infty $.
We consider the evaluation map
$\pi \colon \mathcal{G} \otimes_{ \mathbf{C} } \mathcal{O}_{U} \rightarrow E |_{U}$.
We define $\mathcal{F} \coloneqq {\rm Im} (\pi)$.
By Noether's Lemma (see \cite[Chapter 5 \S 6]{GR}), $\mathcal{F}$ is a coherent sheaf on $U$.

\begin{claim}
For any $x \in U$ and any positive integer $k$,
$$
\mathcal{F}_x + E(h)_x \cap m_{x}^{k} \cdot  E_{(x)} = E(h)_x
$$
holds, where $m_x$ is a maximal ideal of $\mathcal{O}_x$.
\label{Krull}
\end{claim}
We postpone the proof of Claim \ref{Krull} and conclude the proof of Theorem \ref{coh}.
We fix $x \in U$. By the Artin-Rees lemma, there exists a positive integer $l$ such that 
$$
E(h)_x \cap m_{x}^{k} \cdot  E_{(x)} = m_{x}^{k-l} ( E(h)_x \cap m_{x}^{l} \cdot  E_{(x)})
$$
holds for any $k > l$.
Therefore by Claim \ref{Krull}, we have
$$
E(h)_x = \mathcal{F}_x + E(h)_x \cap m_{x}^{k} \cdot  E_{(x)} 
\subset
\mathcal{F}_x + m_{x} \cdot E(h)_x
\subset
E(h)_x.
$$
By Nakayama's lemma, we obtain $E(h)_x = \mathcal{F}_x$, which completes the proof.
\end{proof}
We now prove Claim \ref{Krull}.
\begin{proof}
It is easy to check that $ \mathcal{F}_x + E(h)_x \cap m_{x}^{k} \cdot  E_{(x)}  \subset E(h)_x $,
therefore we show that $E(h)_x  \subset \mathcal{F}_x + E(h)_x \cap m_{x}^{k} \cdot  E_{(x)} $.

We take $f = \sum_{i} f_i e_i \in E(h)_x$ then there exists an open neighborhood $W \subset \subset U$ near $x$ such that $f_i$ is a holomorphic function on $W$ and 
$\int_{W} |f|^{2}_{h} dV_{\omega} < +\infty$.
Let $\rho$ be a cut-off function on $W$.
We note that $\bar{\partial} (\rho f)$ is an $E$-valued $( 0 , 1)$ smooth form such that 
$\int_{X} | \rho f |^{2} _{\omega , h} dV_{\omega} < +\infty$.
We define the plurisubharmonic function $\varphi_k$ to be
$\varphi_k (z)= (n + k)\log |z-x|^{2} + C|z|^2 $ such that
$$
\sqrt{-1} \Theta_{E , h} + \deldel \varphi_k \otimes \textit{Id}_{E} 
\ge \omega \otimes \textit{Id}_{E} \mbox{\,on } X \setminus Z 
\mbox{ in the sense of Nakano,}
$$
where $C$ is some positive constant.
Since $\rho$ is a cut-off function, we obtain
$\int_{X} | \bar{\partial} (\rho f) |^{2} _{\omega , h} e^{-\varphi_k}dV_{\omega} < +\infty$.

Since $X \setminus Z$ is complete by \cite[Th\'eor\`eme 0.2]{Dem82},
there exists 
an $E$-valued $(0 , 0)$ form
$F = \sum_{i}F_i e_i$ 
on $X \setminus Z $ such that 
$$
\int_{X \setminus Z} | F |^{2} _{h} e^{-\varphi_k} dV_{\omega} \le 
\int_{X} | \bar{\partial} (\rho f) |^{2} _{ \omega , h} e^{-\varphi_k} dV_{\omega} < +\infty
\mbox{ \,\, and \,\,} 
\bar{\partial} F= \bar{\partial} (\rho f)
 $$
by Theorem \ref{estimate}.
Here we may regard $\bar{\partial} (\rho f)$ as an $(n , 1)$ form on $X$ by using
$dz^1 \wedge \cdots \wedge dz^n$.

Let $G \coloneqq \rho f - F = \sum_{i}G_i e_i$, which is an $E$-valued $ (0 , 0)$
form on $X \setminus Z$. We obtain 
$$ 
\int_{X \setminus Z} |G|^{2}_{h} dV_{\omega} < +\infty
\mbox{ \,\, and \,\,} 
\bar{\partial} G = 0.
$$
By Lemma \ref{keylemma}
we have $\sum_{i}\int_{U \setminus Z} |G_i|^{2} dV_{\omega} < +\infty$,
therefore $G_i$ extends on $U$ and $G_i$ is holomorphic on $U$
 by the Riemann extension theorem.
Hence we obtain $G \in \mathcal{G}$ and $G_x \in \mathcal{F}_x$.

Let $W'$ be the set of interior points in $\{ z \in U : \rho(z) =1 \}$, then we have $F = f - G$ on $W' \setminus Z$. 
Then $F$ extends on $W'$ and $F$ is holomorphic 
on $W'$. It is obvious that $F_x \in E(h)_x$ from $f_x \in E(h)$ and
$G_x \in \mathcal{F}_x \subset E(h)_x$.
By $\int_{X \setminus Z} | F |^{2} _{h} e^{-\varphi_k} dV_{\omega} < + \infty$ and Lemma 
\ref{keylemma},
we have 
$$\sum_{i} \int_{ W' } | F_i |^{2} e^{-(n + k)\log |z-x|^{2}} dV_{\omega} < + \infty . $$
Therefore we obtain $(F_{i})_{x} \in m_{x}^k$ and $ F_x \in m_{x}^{k} \cdot E_{(x)}$.

Thus we have $f_x = G_x + F_x \in  \mathcal{F}_x + E(h)_x \cap m_{x}^{k} \cdot  E_{(x)}$,
which completes the proof of Claim \ref{Krull}.
\end{proof}

\section{Vanishing theorems and injectivity theorems}

Let $(X , \omega)$ be a compact \kah manifold and $(E , h)$ be a holomorphic vector bundle  with a singular Hermitian metric on $X$.
We assume the conditions (1) -- (3) in Theorem \ref{coh}.
We will denote $Y \coloneqq X \setminus Z$.
By \cite[Section 3]{Fuj07}, there exists a complete \kah form $\omega'$ on $Y$ such that $ \omega' \ge \omega$ on $Y$.
We study the cohomology group $H^q( X , K_X \otimes E(h) ) $.
\begin{thm}

\label{cohomology}
Under the conditions stated above,
we obtain the following isomorphism
$$
H^q( X , K_X \otimes E(h) ) 
\cong 
\frac{L^{2}_{n , q}( Y , E ) _{\omega' , h} \cap {\rm Ker} \bar{\partial} }{{\rm Im} \bar{\partial} }
$$
for any $q \ge 0$.
\end{thm}

\begin{proof}
The proof will be divided into three steps.

Step 1 Set up.

Let $\mathcal{U} = \{ U_j \}_{j \in \Lambda}$ be a finite Stein cover of $X$.
By Theorem \ref{coh}, the sheaf cohomology $H^q( X , K_X \otimes E(h) ) $ is
isomorphic to the \v{C}ech cohomology $H^q( \mathcal{U} , K_X \otimes E(h) ) $.
If necessarily we take $U_j$ small enough,
we may assume that there exist 
a Stein open set $V_j$, 
a smooth plurisubharmonic function $\varphi_j$ on $V_j$
and a positive number $C_j > 0$
such that 
\begin{enumerate}
\item $ U_j \subset \subset V_j$,
\item $C_{j}^{-1} < e^{-\varphi_j} < C_j $ on $V_j$, and 
\item $\sqrt{-1}\Theta_{E , h} + \deldel \varphi_j \ge \omega' \otimes \textit{Id}_E$ on $V_j \setminus Z$
\end{enumerate}
for any $j \in \Lambda$.
We put $U_{i_{0} i_{1} \dots i_{q} } \coloneqq U_{i_0} \cap U_{i_1} \cap \dots \cap U_{i_q}$, which is a Stein open set.

With these conditions above, it is easily to check the following two claims.

\begin{claim}\cite[Remark 2.19]{Fuj07}
\label{inequality}
For any $E$-valued $(n , q)$ form $u$ on $Y$ with measurable coefficients, 
$ | u |^2_{\omega'  , h } dV_{ \omega'  } 
\le
| u |^2_{\omega  , h } dV_{ \omega} 
$
holds. If $q = 0$,
$ | u |^2_{\omega'  , h } dV_{ \omega'  } 
=
| u |^2_{\omega  , h } dV_{ \omega} 
$
holds.
\end{claim}

\begin{claim}
\label{modl2}
For any $q \ge 1$ and any $g \in L^{2}_{ n , q }( U_{i_{0} i_{1} \dots i_{q} } \setminus Z  , E )_{ \omega' , h}$ such that $\bar{\partial} g = 0$,
there exists $f \in L^{2}_{ n , q-1 }( U_{i_{0} i_{1} \dots i_{q} } \setminus Z  , E )_{ \omega' , h}$  such that 
$\bar{\partial}  f =  g $ and 
$$
\int_{U_{i_{0} i_{1} \dots i_{q} } \setminus Z  } | f |^2_{\omega'  , h } dV_{ \omega'  }  \le 
C'^2
\int_{U_{i_{0} i_{1} \dots i_{q} } \setminus Z } | g|^2_{\omega'  , h } dV_{ \omega'  },
$$
where $C' \coloneqq max_{ i \in \Lambda } C_i$.
\end{claim}

Step 2 Construction of a homomorphism from \v{C}ech cohomology to Dolbeault cohomology.

We fix $c = \{ c_{i_{0} i_{1} \dots i_{q} }\} \in H^q( \mathcal{U} , K_X \otimes E(h) )$.
By the definition of \v{C}ech cohomology, we have 
\begin{enumerate}
\item $c_{i_{0} i_{1} \dots i_{q} } \in H^0( U_{i_{0} i_{1} \dots i_{q} } , K_X \otimes E(h) ) $ and  
\item $ \delta c \coloneqq \sum_{k = 0}^{q+1} (-1)^k c_{ i_{0} i_{1} \dots \check{i_k} \dots i_{q+1} }    |_{U_{i_{0} i_{1} \dots i_{q+1}}  } =0$.
\end{enumerate}

Let $\{ \rho_{i} \}_{i \in \Lambda}$ be a partition of unity subordinate to $\mathcal{U}$.
We define an $E$-valued form $b_{i_{0} i_{1} \dots i_{k} } $ inductively by 
$$
b_{i_{0} i_{1} \dots i_{q-1} } \coloneqq \sum_{j \in \Lambda} \rho_{j} c_{j i_{0} i_{1} \dots i_{q-1} } 
 \mbox{\,\, and \,\,} 
b_{i_{0} i_{1} \dots i_{k} } \coloneqq \sum_{j \in \Lambda} \rho_{j} 
\bar{\partial} b_{j i_{0} i_{1} \dots i_{k} } .
$$
We point out
 $$
 \delta \{ b_{i_{0} i_{1} \dots i_{q-1} }  \} = c,
 \, \,
\delta \{ \bar{\partial} b_{i_{0} i_{1} \dots i_{q-1} }  \}  =0, 
\mbox{\,\, and \,\,}
\delta \{ b_{i_{0} i_{1} \dots i_{k} }  \}  = \{ \bar{\partial} b_{i_{0} i_{1} \dots i_{k+1} }  \}
$$
hold (see \cite[Lemma 4.2] {Mat16}).

Therefore we obtain $\bar{\partial} b_{i_0} |_{U_{i_0} \setminus Z} $,
which is an $E$-valued $(n , q)$
$\bar{\partial}$-closed form on $U_{i_0} \setminus Z$.
Since we have
$$
\delta \{ \bar{\partial} b_{i_0} \}= 0 
\mbox{\,\, and \,\,}
\int_{U_{i_{0} } \setminus Z  } | \bar{\partial} b_{i_0} |^2_{\omega'  , h } dV_{ \omega'  }  
\le
\int_{U_{i_{0} } } | \bar{\partial} b_{i_0} |^2_{\omega  , h } dV_{ \omega} 
<
+\infty
$$ 
by Claim \ref{inequality}, we can define $\alpha (c) \coloneqq \{ \bar{\partial} b_{i_0} \} \in
L^{2}_{n , q}( Y , E ) _{\omega' , h} \cap {\rm Ker} \bar{\partial}  $.
By the above construction, we obtain the homomorphism 
$$ \alpha \colon H^q( \mathcal{U} , K_X \otimes E(h) )  \rightarrow \frac{L^{n , q}_2( Y , E ) _{\omega' , h} \cap {\rm Ker} \bar{\partial} }{{\rm Im} \bar{\partial} } .$$

Step 3 Construction of a homomorphism from Dolbeault cohomology to \v{C}ech cohomology.

We fix $u \in L^{2}_{n , q}( Y , E ) _{\omega' , h} \cap {\rm Ker} \bar{\partial} $ and define 
$D \coloneqq \int_{Y } | u |^2_{\omega'  , h } dV_{ \omega'  } < + \infty $.
By Claim \ref{modl2}, there exists $v_{i_0} \in L^{2}_{n , q-1}( U_{i_0} \setminus Z, E ) _{\omega' , h} $ such that 
$$
\bar{\partial} v_{i_0} = u |_{U_{i_0} \setminus Z }
\mbox{\,\, and \,\,}
\int_{U_{i_0} \setminus Z } | v_{i_0} |^2_{\omega'  , h } dV_{ \omega'  }  \le C'^{2} D.
$$
We put $u^1 \coloneqq \delta \{ v_{i_0} \}$.
From $\bar{\partial} u^1 = 0$, there exists 
$v_{i_{0} i_{1} } \in L^{2}_{n , q-1}( U_{ i_{0} i_{1} } \setminus Z, E ) _{\omega' , h} $
such that 
$$
\bar{\partial} v_{i_{0} i_{1} }  = u^{1}_{ i_{0} i_{1}  } 
\mbox{\,\, and \,\,}
\int_{U_{i_{0} i_{1}} \setminus Z } | v_{i_{0} i_{1} } |^2_{\omega'  , h } dV_{ \omega'  }  \le 2C'^{2} D
$$
by Claim \ref{modl2}.
We put $u^2 \coloneqq \delta \{ v_{i_{0} i_{1}} \}$ and we have $\bar{\partial} u^2 = 0$.

By repeating this procedure, we obtain 
$v_{i_{0} i_{1} \dots i_{q-1} } \in L^{2}_{n , 0}( U_{ i_{0} i_{1} \dots i_{q-1} } \setminus Z, E ) _{\omega' , h}$
and $u^q = \delta \{v_{i_{0} i_{1} \dots i_{q-1} }  \} $. 
By $\bar{\partial} u^{q}_{i_{0} i_{1} \dots i_{q} } = 0$,
$u^{q}_{i_{0} i_{1} \dots i_{q} } $ is a holomorphic $E$-valued $(n , 0)$ form on
$U _{ i_{0} i_{1} \dots i_{q} } \setminus Z$.
Since we obtain
$$
\int_{ U _{ i_{0} i_{1} \dots i_{q} } \setminus Z } | u^{q}_{i_{0} i_{1} \dots i_{q} }  |^2_{\omega , h } dV_{ \omega } 
= \int_{ U _{ i_{0} i_{1} \dots i_{q} } \setminus Z } | u^{q}_{i_{0} i_{1} \dots i_{q} }  |^2_{\omega'  , h } dV_{ \omega'  }  \le q! C'^{2} D < + \infty
$$ 
by Claim \ref{inequality},
$u^{q}_{i_{0} i_{1} \dots i_{q} }  |_{U _{ i_{0} i_{1} \dots i_{q} } \setminus Z} $
extends on $U _{ i_{0} i_{1} \dots i_{q} }$
and
$u^{q}_{i_{0} i_{1} \dots i_{q} }  |_{U _{ i_{0} i_{1} \dots i_{q} } \setminus Z} $
is a holomorphic $E$-valued $(n , 0)$ form on $U _{ i_{0} i_{1} \dots i_{q} }$
by the Riemann extension theorem and Lemma \ref{keylemma}.
Therefore we can define $\beta (u) \coloneqq \{ u^{q}_{i_{0} i_{1} \dots i_{q} }  |_{U _{ i_{0} i_{1} \dots i_{q} } \setminus Z}   \} \in H^q( \mathcal{U} , K_X \otimes E(h) )$.
By the above construction, we obtain the homomorphism
$$
 \beta \colon \frac{L^{2}_{n , q}( Y , E ) _{\omega' , h} \cap {\rm Ker} \bar{\partial} }{ {\rm Im} \bar{\partial} }
\rightarrow H^q( \mathcal{U} , K_X \otimes E(h) ) .
$$

It is easily to check $\alpha$ and $\beta$ induce the isomorphism in Theorem \ref{cohomology}.
\end{proof}

We finish this section by showing Theorem \ref{vanishing} and \ref{injective}.

{ \it Proof of Theorem \ref{vanishing}.}
It is easy by Theorem \ref{estimate} and Theorem \ref{cohomology}.

{ \it Proof of Theorem \ref{injective}.}
By Theorem $\ref{coh}$, $K_X \otimes E(h)$ is a coherent sheaf on $X$.
Therefore by the argument of \cite[Claim 1]{Fuj07}, Theorem \ref{harmonic} and Theorem \ref{cohomology},
we obtain $\overline{ {\rm Im} \bar{\partial} } =  {\rm Im} \bar{\partial} $,
$ \overline{{\rm Im} \bar{\partial}^{*}_{\omega' , h}} = {\rm Im} \bar{\partial}^{*}_{\omega' , h}$ and 
$
H^q( X , K_X \otimes E(h) ) 
\cong
\mathcal{H}^{n,q}( Y , E )
$.
Similarily, we obtain $H^q( X , K_X \otimes L \otimes E(h)  ) \cong
\mathcal{H}^{n,q}( Y , L \otimes E )$.
By \cite[Claim 2]{Fuj07},
the multiplication homomorphism 
$\times s \colon \mathcal{H}^{n,q}( Y , E ) \rightarrow \mathcal{H}^{n,q}( Y , L \otimes E ) $
is well-defined and injective, which completes the proof.

\section{Applications}
\subsection{Hosono's example}
In this subsection, we study a singular Hermitian metric induced by holomorphic sections, proposed by Hosono \cite[Chapter 4]{Hos17}.

Let $s_1, \dots, s_N \in H^0( X , E )$ be holomorphic sections such that $E_{y}$
is generated by $s_1(y), \dots, s_N(y)$ for a general point $y$.
For any point $x \in X$, we take a local coordinate $( U ; z_1 , \dots , z_n)$
near $x$ and take a local holomorphic frame $e_1, \dots, e_r$ of $E$ on $U$.
Write $s_i = \sum_{1 \le j \le r} f_{ij}e_j$, where $f_{ij}$ are holomorphic functions on $U$.
A singular Hermitian metric $h$ induced by $s_1, \dots, s_N$ is given by
$$
h^{-1}_{jk} \coloneqq \sum_{1 \le i \le N} \bar{ f_{ij} }f_{ik} .
$$
By \cite[Example 3.6 and Proposition 4.1]{Hos17}, $h$ is positively curved and $E(h)$ is a coherent sheaf.
Hosono pointed out that
we can easily calculate the curvature of $h$ in the case $N=r$. 

\begin{lem}
\label{Nr}
In the case $N=r$, there exists a proper analytic subset $Z$ such that 
$\sqrt{-1} \Theta_{E , h} = 0$ on $X \setminus Z$. 
In particular we obtain $\sqrt{-1} \Theta_{E , h} \ge 0$ on $X \setminus Z$ in the sense of Nakano.
\end{lem}
\begin{proof}
We take a finite Stein open covering $\{ U_i \}_{i \in \Lambda}$.
Under the condition stated above, the $r \times r$ matrix $A^{(i)}$ on $U_i$ is defined by 
$$
A^{(i)}_{jk} = f_{jk}.
$$
Set  $Z_i \coloneqq \{  z \in U_i \colon \text{rank} \, A^{(i)}(z) < r \}$ and
$W = \{ z \in X \colon h \mbox{ is not smooth at } z \}$.
Write $Z \coloneqq \cup_{i \in \Lambda} Z_i \cup W$, which is a proper analytic subset.

By an easy computation, we have
$$
\sqrt{-1}\Theta_{E, h}
= 
\sqrt{-1} \bar{\partial} ( \bar{h}^{-1} \partial \bar{h})
=
\sqrt{-1}( \partial \bar{\partial} \bar{h}^{-1} -  \partial \bar{h}^{-1} \bar{h} \bar{\partial} \bar{h}^{-1} ) \bar{h}.
$$
For any $z \in X \setminus Z$, we may assume $f_{ij}(z) = \delta_{ij}$.
From $\bar{h}^{-1}_{jk} = \sum_{1 \le i \le r}  f_{ij} \bar{ f_{ik} }  $, we have
$$
\partial \bar{h}^{-1}_{jk}(z) = \partial f_{kj}(z)
\mbox{\,\, and \,\,}
\bar{ \partial } \bar{h}^{-1}_{jk}(z) = \bar{ \partial } \bar{ f_{jk} }(z).
$$
Thus, we obtain
$$
(\partial \bar{h}^{-1} \bar{h} \bar{\partial} \bar{h}^{-1} )_{jk}(z)
=
\sum_{1 \le i \le r} \partial f_{ij}  \bar{ \partial } \bar{  f_{ik} }(z)
=
\partial \bar{\partial} \bar{h}^{-1}_{jk}(z),
$$
which completes the proof.
\end{proof}

By Lemma \ref{Nr} and Theorem \ref{vanishing}, we obtain the following corollary.

\begin{cor}
\label{vanishing2}
Let $(L , h_L)$ be a holomorphic line bundle with a singular Hermitian metric.
We assume there exist a proper analytic subset $Z$ and a positive number $\epsilon > 0$
such that $h_L$ is smooth on $X \setminus Z$ and $\sqrt{-1} \Theta_{L , h_L } \ge \epsilon \omega$ on $X$.

Then, 
$H^q(X , K_X \otimes L \otimes E( hh_L)) = 0$
holds for all $q \ge 1$ for any holomorphic vector bundle $E$ and a singular Hermitian metric $h$ induced by $s_1 \cdots s_r \in H^0(X , E)$.

In particular $H^q(X , K_X \otimes L \otimes E( h)) = 0$ holds for all $q \ge 1$ if $L$ is ample.
\end{cor}
We point out that such a metric $h_L$ on $L $ as in Lemma \ref{vanishing2} always exists if $L$ is big.

Now, we introduce Hosono's example \cite[Example 4.4]{Hos17}.
Set $X=\mathbb{C}^2$ and let $E = X \times \mathbb{C}^2$ be a trivial rank two bundle.
We choose sections $s_1 = e_1$ and $s_2 = z e_1 + w e_2$.
Then the singular Hermitian metric $h_{E}$ induced by $s_1, s_2$ can be written by
$$
h_E = \frac{1}{|w|^2}
\begin{pmatrix}
|w|^2 & -w \bar{z} \\
-z \bar{w} & |z|^2 + 1
\end{pmatrix}
.
$$
 Hosono proved the following theorem
 by calculating the standard approximation by convolution of $h_E$.
\begin{thm}\cite[Theorem 1.2]{Hos17}
The standard approximation defined by convolution of $h_E$ 
does not have uniformly bounded curvature from below in the sense of Nakano.
\end{thm}

Therefore, we can not apply the vanishing theorem of \cite{Cat98} to this example.
However we can apply Corollary \ref{vanishing2}  to this example.
Thus our results are new results.

\begin{rem}
We ask whether there exists a proper analytic subset $Z$ such that 
$\sqrt{-1} \Theta_{E , h} \ge 0$ on $X \setminus Z$ in the sense of Nakano
for any singular Hermitian metric $h$ induced by $s_1 \cdots s_N \in H^0(X , E)$
in the case $N > r$. 
This calculation is very complicated and this question is open, but it is likely that the answer is "No".
\end{rem}


\subsection{Big vector bundles}
We review some of the standard facts on big 
vector bundles.
\begin{dfn} \cite[Section 2]{BKKMSU}
Let $X$ be a smooth projective variety and $E$ be a holomorphic vector bundle.
The {\it base locus} of $E$ is defined by
$$
{\rm Bs}(E) \coloneqq \{ x \in X \colon H^0(X , E ) \rightarrow E_x \mbox{ is not surjective} \},
$$
and the {\it stable base locus} of $E$ is defined by
$$
\mathbb{B}(E) \coloneqq \bigcap_{m > 0} {\rm Bs}({\rm Sym}^{m}E),
 $$ 
where $Sym^{m}(E)$ is the $m$-th symmetric power of $E$.

Let $A$ be an ample line bundle.
We define
the {\it argumented base locus} of $E$ by
$$
\mathbb{B}_{+}(E) = \bigcap_{ p / q \in \mathbb{Q}_{>0}} \mathbb{B}( {\rm Sym}^{q} E \otimes A^{p}).
$$
\end{dfn}
We point out $\mathbb{B}_{+}(E)$ do not depend on the choice of the ample line bundle $A$ by \cite[Remark 2.7]{BKKMSU}.

\begin{dfn} \cite[Theorem 1.1 and Section 6]{BKKMSU}
\begin{enumerate}
\item A vector bundle $E$ is said to be {\it L-big}
if the tautological bundle
$\mathcal{O}_{ \mathbb{P}(E)}  (1)$ on $\mathbb{P}(E)$ is big.

\item A vector bundle $E$ is said to be {\it V-big}
if $\mathbb{B}_{+}(E) \neq X$.
\end{enumerate}
\end{dfn}
We note that if $E$ is V-big then it is L-big as well by \cite[Corollary 6.5]{BKKMSU}.
We will denote by $\pi \colon \mathbb{P}(E) \rightarrow X $ the canonical projection and
by $\tilde{\omega} $ a \kah form on $\mathbb{P}(E)$.
Inayama communicated to the author 
the following lemma.
\begin{lem}
\label{inayama}
Let $E$ be an L-big vector bundle and $\tilde{h}$ be a singular Hermitian metric on 
$\mathcal{O}_{ \mathbb{P}(E)}(1)$.
We assume that there exist a positive number $\epsilon > 0$ and a proper analytic subset $\tilde{Z} \subset \mathbb{P}(E)$
such that $\tilde{h}$ is smooth on $\mathbb{P}(E) \setminus \tilde{Z}  $,
$\pi ( \tilde{Z} ) \neq X$, and
$\sqrt{-1} \Theta_{ \mathcal{O}_{ \mathbb{P}(E) }(1) , \tilde{h} } \ge \epsilon
\tilde{\omega} \otimes id_{ \mathcal{O}_{ \mathbb{P}(E)}(1) }$.

Then $\tilde{h}$ induces a singular Hermitian metric $h_m$ on
${\rm Sym}^{m}(E) \otimes \det E $
such that
\begin{enumerate}
\item $ h_m $ is smooth on $X \setminus \pi ( \tilde{Z} )$,
\item $h_m$ is a positively curved singular Hermitian metric, and
\item $\sqrt{-1}\Theta_{{\rm Sym}^{m}(E) \otimes \det E, h_m } \ge \epsilon \omega \otimes Id_{{\rm Sym}^{m}(E) \otimes \det E}  $ on $X \setminus \pi ( \tilde{Z} )$ in the sense of Nakano.
\end{enumerate}

\end{lem}
\begin{proof}
From ${\rm Sym}^{m}(E) \otimes \det E 
= \pi_{*}( K_{ \mathbb{P}(E) / X} \otimes \mathcal{O}_{ \mathbb{P}(E)}(m + r)) $,
${\rm Sym}^{m}(E) \otimes \det E $ can be endowed with the $L^2$ metric $h_m$
with respect to $\tilde{h}$.
Therefore by the argument of \cite[Theorem 1.2, Theorem 1.3, and Section 4]{Ber07},
(1) and (3) are proved.
By \cite{HPS17} and \cite{PT14}, (2) is proved.
\end{proof}

\begin{rem}
By \cite[Proposition 3.2]{BKKMSU},
$\pi \Bigl( \mathbb{B}_{+}( \mathcal{O}_{ \mathbb{P}(E)} (1) ) \Bigr)= \mathbb{B}_{+}(E)$
holds.
Therefore if $E$ is V-big, such a metric $\tilde{h}$ on $\mathcal{O}_{ \mathbb{P}(E)}(1)$ as the assumption of Lemma \ref{inayama} always exists.
\end{rem}

Thus, we can apply Theorem \ref{vanishing} to $( {\rm Sym}^{m}(E) \otimes \det E , h_m)$
and we have the following corollary.
\begin{cor}
Under the conditions stated in Lemma \ref{inayama},
$H^q(X , K_X \otimes ({\rm Sym}^{m}(E) \otimes \det E) (h_m)) = 0$ holds
for any $m \ge 0$ and $q \ge 1$.
\end{cor}
This corollary is a generalization of Griffiths' vanishing theorem in \cite{Gri}.

\section{On Ohsawa's vanishing theorem}
We use the results of \cite{Oh88}.
Let $Y$ be a complete \kah manifold, $\omega'$ be a \kah form
and $(E , h)$ be a vector bundle with a smooth Hermitian metric.
Let $\tau$ be a smooth semipositive $(1,1)$ form on $Y$.
Write
$$
L^{2}_{ n , q }( Y , E )_{ \tau , h} \coloneqq 
 \{ f \in L^{2}_{ n , q }( Y , E )_{ \omega' +\tau , h} \,;\,\,
\lim_{\epsilon \downarrow 0}\int_{Y} | f |^2_{\ \epsilon \omega' + \tau, h } dV_{ \epsilon \omega' + \tau} < + \infty \}.
$$
By \cite[Proposition 2.4]{Oh88}, $\lim_{\epsilon \downarrow 0}\int_{Y} | f |^2_{\ \epsilon \omega' + \tau, h } dV_{ \epsilon \omega' + \tau}$ and $L^{2}_{ n , q }( Y , E )_{ \tau , h}$ do not depend on the choice of the metric $\omega'$.
We use Ohsawa's $L^2$ estimate.
\begin{thm}\cite[Theorem 2.8]{Oh88}
\label{ohsawal2}
Under the conditions stated above,
we also assume that $\sqrt{-1} \Theta_{E , h} -  \tau \otimes \textit{Id}_{E}  \ge 0$ on $Y$.
For any $q \ge 1$ and $f \in L^{2}_{ n , q }( Y , E )_{ \tau , h} $ such that $\bar{\partial} f = 0$,
there exists $g \in L^{2}_{ n , q-1 }( Y , E )_{ \tau , h}  $ such that 
$\bar{\partial} g = f$.

\end{thm}

Now we prove Theorem \ref{ohsawa}.

\begin{proof}
We take a complete \kah form $\omega'$ on $Y \coloneqq X \setminus Z$ as in Section 4.
The proof of Theorem \ref{ohsawa} is similar to \cite[Theorem 3.1]{Oh88} and Theorem \ref{cohomology} with a slight modification.

Let $\mathcal{U} = \{ U_j \}_{j \in \Lambda}$ be a finite Stein cover of $W$.
By Theorem \ref{coh} and the Grauert direct image theorem, the sheaf cohomology $H^q( W  , \pi_{*} (K_X \otimes E(h) ) ) $ is
isomorphic to the \v{C}ech cohomology $H^q( \mathcal{U} , \pi_{*} (K_X \otimes E(h) ) ) $.
We point out the following claim.
\begin{claim}\cite[Lemma 3.2]{Oh88}
\label{ohsawalem}
For any form $g$ on $W$, 
$| \pi^{*}g (x)|_{ \omega + \pi^{*}\sigma} \le | g (\pi(x) )|_{ \sigma}  $ holds at any $x \in X$.
\end{claim}

We fix $c = \{ c_{i_{0} i_{1} \dots i_{q} }\} \in H^q( \mathcal{U} , \pi_{*} (K_X \otimes E(h) ) )$.
By the definition of \v{C}ech cohomology, we have 
\begin{enumerate}
\item $c_{i_{0} i_{1} \dots i_{q} } \in H^0( U_{i_{0} i_{1} \dots i_{q} } , \pi_{*} (K_X \otimes E(h) ) )  = H^0( \pi^{-1} (U_{i_{0} i_{1} \dots i_{q} } ), K_X \otimes E(h) ) $ and  
\item $ \delta c \coloneqq \sum_{k = 0}^{q+1} (-1)^k c_{ i_{0} i_{1} \dots \check{i_k} \dots i_{q+1} }    |_{ \pi^{-1} (U_{i_{0} i_{1} \dots i_{q+1}} ) }=0$.
\end{enumerate}

Let $\{ \rho_j\}_{j \in \Lambda} $ be a partition of unity of $\mathcal{U} $.
Based on Section 4,
we define an $E$-valued form $b_{i_{0} i_{1} \dots i_{k} } $ inductively by 
$$
b_{i_{0} i_{1} \dots i_{q-1} } \coloneqq \sum_{j \in \Lambda} \pi^{*}(\rho_{j}) c_{j i_{0} i_{1} \dots i_{q-1} } 
 \mbox{\,\, and \,\,} 
b_{i_{0} i_{1} \dots i_{k} } \coloneqq \sum_{j \in \Lambda} \pi^{*}( \rho_{j} )
\bar{\partial} b_{j i_{0} i_{1} \dots i_{k} } .
$$
We point out 
 $$
 \delta \{ b_{i_{0} i_{1} \dots i_{q-1} }  \} = c,
 \, \,
\delta \{ \bar{\partial} b_{i_{0} i_{1} \dots i_{q-1} }  \}  =0, 
\mbox{\,\, and \,\,}
\delta \{ b_{i_{0} i_{1} \dots i_{k} }  \}  = \{ \bar{\partial} b_{i_{0} i_{1} \dots i_{k+1} }  \}
$$
hold.

Therefore we obtain $\bar{\partial} b_{i_0} |_{\pi^{-1}(U_{i_0}) \setminus Z} $,
which is an $E$-valued $(n , q)$
$\bar{\partial}$-closed form on $\pi^{-1}(U_{i_0}) \setminus Z$.
By Claim \ref{ohsawalem}, $|\bar{\partial} (\pi^{*} \rho_{j} )|_{\epsilon \omega + \pi^{*} \sigma}$ are bounded above by $|\bar{\partial} (\rho_{j} )|_{\sigma}$ for any $\epsilon > 0$ and $| c_{i_{0} i_{1} \dots i_{q} } |^{2}_{\epsilon \omega + \pi^{*} \sigma} dV_{\epsilon \omega + \pi^{*} \sigma}$ are independent of $\epsilon$ by Claim \ref{inequality}.
Therefore we have
$\delta \{ \bar{\partial} b_{i_0} \}= 0 $ and
\begin{align*}
\int_{\pi^{-1}(U_{i_0}) \setminus Z } | \bar{\partial} b_{i_0} |^2_{\epsilon \omega' + \pi^{*} \sigma  , h } dV_{ \epsilon \omega' + \pi^{*} \sigma }  
&\le
\int_{\pi^{-1}(U_{i_0})  } | \bar{\partial} b_{i_0} |^2_{\epsilon \omega + \pi^{*} \sigma  , h } 
dV_{ \epsilon \omega + \pi^{*} \sigma} \\
&\le
\lim_{\epsilon \downarrow 0}\int_{\pi^{-1}(U_{i_0}) } | \bar{\partial} b_{i_0} |^2_{\epsilon \omega + \pi^{*} \sigma , h } 
dV_{ \epsilon \omega + \pi^{*} \sigma} \\
&<
+\infty
\end{align*} 
for any $\epsilon > 0 $ by Claim \ref{inequality}.
Thus, we may regard $ \{ \bar{\partial} b_{i _{0}} \}$ as an element of 
$L^{2}_{ n , q }( Y , E )_{ \sigma , h}$ and denote by $b \coloneqq \bar{\partial} b_{i _{0}}$.
By Theorem \ref{ohsawal2}, there exists $a \in L^{2}_{ n , q-1 }( Y , E )_{ \sigma , h}$
such that 
$$
\bar{\partial} a = b
 \mbox{\,\, and \,\,} 
\lim_{\epsilon \downarrow 0} \int_{Y \setminus Z  } | a |^2_{\epsilon \omega' + \pi^{*} \sigma , h } dV_{ \epsilon \omega' + \pi^{*} \sigma   }  < + \infty.
$$
Write $d^{1}_{i_{0}} \coloneqq b_{i_{0}} - a \in L^{2}_{ n , q-1 }(\pi^{-1}(U_{i_{0}}) \setminus Z , E )_{ \sigma , h}$
and $d^1 \coloneqq \{d^{1}_{i_{0}} \}$.
We point out 
$$
 \delta d^1 = \delta\{ b_{i_{0}}\} =\{ \bar{\partial} b_{ i_{0} i_{1} } \}
  \mbox{\,\, and \,\,} 
\bar{\partial} d^1  = 0.
 $$
By Theorem \ref{ohsawal2}, there exists $a_{i_{0}} \in L^{2}_{ n , q-2}( \pi^{-1}(U_{i_{0}}) \setminus Z  , E )_{ \sigma , h}$ such that 
$$
\bar{\partial} a_{i_{0}} =   d^{1}_{i_{0}} 
 \mbox{\,\, and \,\,} 
\lim_{\epsilon \downarrow 0} \int_{U_{i_{0}} \setminus Z  } | a_i |^2_{\epsilon \omega' + \pi^{*} \sigma , h } dV_{ \epsilon \omega' + \pi^{*} \sigma   }  < + \infty.
$$
Write 
$d^{2}_{i_{0}i_{1}} \coloneqq b_{i_{0} i_{1}} - a_{i_{0}} + a_{i_{1}} \in L^{2}_{ n , q-2}(\pi^{-1}(U_{i_{0} i_{1}} ) \setminus Z , E )_{ \sigma , h}$
and $d^2 \coloneqq \{d^{2}_{i_{0}i_{1}} \}$.
We point out 
$$
 \delta d^2= \delta\{b _{i_{0} i_{1}}\}=\{ \bar{\partial} b_{ i_{0} i_{1}i_{2} } \}
  \mbox{\,\, and \,\,} 
\bar{\partial} d^2  = 0.
$$
By repeating this procedure, we obtain 
$d^{q-1} _{i_{0} i_{1} \dots i_{q-1} } \in L^{2}_{ n , 0 }(\pi^{-1}(U _{i_{0} i_{1} \dots i_{q-1} } ) \setminus Z , E )_{ \sigma , h} $
and $d^{q-1} \coloneqq \{d^{q-1}_{i_{0} i_{1} \dots i_{q-1} } \}$ such that
 $$
 \delta d^{q-1} = \delta \{ b_{i_{0} i_{1} \dots i_{q-1} } \}=c
  \mbox{\,\, and \,\,} 
\bar{\partial} d^{q-1} = 0.
$$
We have
\begin{align*}
\int_{\pi^{-1}(U  _{i_{0} i_{1} \dots i_{q-1} } ) \setminus Z} | d^{q-1} _{i_{0} i_{1} \dots i_{q-1} } |^2_{\omega  , h } dV_{ \omega}  
&=
\int_{ \pi^{-1}(U _{i_{0} i_{1} \dots i_{q-1} } ) \setminus Z} | d^{q-1} _{i_{0} i_{1} \dots i_{q-1} } |^2_{\ \omega' +\pi^{*} \sigma , h } dV_{  \omega' + \pi^{*} \sigma  }  \\
&=
\lim_{\epsilon \downarrow 0} \int_{ \pi^{-1}(U _{i_{0} i_{1} \dots i_{q-1} } ) \setminus Z} | d^{q-1} _{i_{0} i_{1} \dots i_{q-1} } |^2_{\epsilon \omega' + \pi^{*} \sigma  , h } dV_{ \epsilon \omega' + \pi^{*} \sigma  }  \\
& < + \infty.
\end{align*}
By Lemma \ref{keylemma} and the Riemann extension theorem, 
$d^{q-1}_{i_{0} i_{1} \dots i_{q-1} } $ extends on 
 $\pi^{-1}(U  _{i_{0} i_{1} \dots i_{q-1} } ) $ and $d^{q-1}_{i_{0} i_{1} \dots i_{q-1} } $ is holomorphic on $\pi^{-1}(U  _{i_{0} i_{1} \dots i_{q-1} } ) $.
Therefore we obtain 
$d^{q-1} _{i_{0} i_{1} \dots i_{q-1} } \in H^0( \pi^{-1} (U_{i_{0} i_{1} \dots i_{q-1} } ), K_X \otimes E(h) ) $
and 
$ \delta d^{q-1} = c$,
which completes the proof.
\end{proof}

\begin{rem}
We ask whether under the assumptions of singular Hermitian metrics as in Theorem \ref{vanishing} - \ref{ohsawa}, we can show higher rank analogies of a generalization of Koll\'ar-Ohsawa type vanishing theorem by Matsumura \cite{Mat16}, an injectivity theorem of 
higher direct images by Fujino \cite{Fuj2},
an injectivity theorem of pseudoeffective line bundles by Fujino and Matsumura \cite{FM}
and so on.
It is likely the answer is "Yes" and the proof may be similar to the original proof with a slight modification.
\end{rem}

\color{black}


\end{document}